\documentclass{amsart}
\usepackage{amssymb,amsmath,bbm}
\usepackage{pdflscape}
\usepackage{mathrsfs}
\usepackage{tikz}
\usepackage{graphicx}
\usepackage{caption}
\usepackage{subcaption}
\usepackage{rotating}
\usepackage{afterpage}
\usepackage{hyperref}
\usepackage{url}

\usetikzlibrary{calc,intersections, arrows, shapes}
\newcommand{\red}{\textcolor{red}}

\newcommand{\CC}{\mathbb{C}}

\newcommand{\RR}{\mathbb{R}}

\newcommand{\VV}{\mathcal{V}}

\newcommand{\rank}{\textup{rank}\,}

\newcommand{\supp}{\textup{supp}}
\newcommand{\conv}{\textup{conv}}

\newcommand{\xx}{\mathbf x}
\newcommand{\yy}{\mathbf y}
\newcommand{\s}{\mathbf s}

\newcommand{\mb}[1]{\boldsymbol #1}

\newtheorem{theorem}{Theorem}[section]

\newtheorem{lemma}[theorem]{Lemma}
\newtheorem{corollary}[theorem]{Corollary}

\theoremstyle{definition}
\newtheorem{definition}[theorem]{Definition}
\newtheorem{example}[theorem]{Example}

\theoremstyle{remark}
\newtheorem{remark}[theorem]{Remark}

\title{The Slack  Realization Space of a Polytope}

\author[Gouveia]{Jo{\~a}o Gouveia}
\address{CMUC, Department of Mathematics,
  University of Coimbra, 3001-454 Coimbra, Portugal}
\email{jgouveia@mat.uc.pt}

\author[Macchia]{Antonio Macchia}
\address{Discrete Geometry Group,  
Freie Universit\"at Berlin, 
Arnimallee 2, 
14195 Berlin, Germany}
\email{macchia.antonello@gmail.com}

\author[Thomas]{Rekha R. Thomas}
\address{Department of Mathematics, University of Washington, Box
  354350, Seattle, WA 98195, USA} \email{rrthomas@uw.edu}

\author[Wiebe]{Amy Wiebe}
\address{Department of Mathematics, University of Washington, Box
  354350, Seattle, WA 98195, USA} \email{awiebe@uw.edu}

\thanks{Gouveia was partially supported by the Centre for Mathematics of the
University of Coimbra -- UID/MAT/00324/2013, funded by the Portuguese
Government through FCT/MEC and co-funded by the European Regional Development Fund through the Partnership Agreement PT2020, Macchia was supported by INdAM, Thomas by
the U.S. National Science Foundation grant DMS-1418728, and Wiebe by NSERC}

\date{\today}

\begin{document}

\begin{abstract}
In this paper we introduce a natural model 
for the realization space of a polytope up to projective equivalence which we call the slack realization space of the polytope. 
The model arises from the positive part of an algebraic variety 
determined by the slack ideal of the polytope. This is a saturated
determinantal ideal that encodes the
combinatorics of the polytope. We also derive a new model of 
the realization space of a polytope from the positive part of the variety of a related ideal.
The slack ideal offers an effective 
computational framework for several classical questions about
polytopes such as rational realizability, non-prescribability of faces, and realizability of
combinatorial polytopes.
\end{abstract}


\maketitle


%
%


\section{Introduction} \label{sec:introduction}

An important focus in the study of polytopes is the investigation of their realization spaces. Given a $d$-polytope $P \subset \RR^d$, its face lattice determines its combinatorial type. A realization space of $P$ is, roughly speaking,  the set of all geometric realizations of the combinatorial type of $P$.
This set, usually defined by fixing an affinely independent set of vertices in every realization of $P$,  is a primary basic semialgebraic set, meaning that it is defined by a finite set of polynomial equations and strict inequalities.

Foundational questions about polytopes such as whether there is a polytope with rational vertices
in the combinatorial class of $P$, whether a combinatorial type has any realization at all
as a convex polytope, or whether faces of a polytope can be freely prescribed, are all questions about realization spaces. In general, many of these questions are hard to settle and there is no straightforward way to answer them by working directly with realization spaces. Each instance of such a question often requires a clever new strategy; indeed, the polytope literature contains many ingenious methods to find the desired answers.

In this paper, we introduce {a model for the realization space of a polytope in a given combinatorial class} modulo projective transformations.  This space arises from the positive part of an algebraic variety called
the {\em slack variety} of the polytope. {An explicit model for the realization space of }the projective equivalence classes of a polytope does not exist in the literature, although several authors have implicitly worked modulo projective transformations \cite{AP17,APT15,R96}. 
Using a related idea, we also construct a {model for the} realization space for a polytope that is rationally equivalent to the classical {model for the} realization space of the polytope.
The ideal giving rise to the slack variety is called the {\em slack ideal} of the polytope and was introduced in \cite{GPRT17}. The slack ideal in turn was inspired by the {\em slack matrix} of a polytope. This is a nonnegative real matrix with rows (and columns) indexed by the vertices (and facets) of the polytope and with $(i,j)$-entry equal to the slack of the $i$th vertex in the $j$th facet inequality. Each vertex/facet representation of a $d$-polytope $P$ gives rise to a slack matrix $S_P$ of rank $d+1$.
Slack matrices have found remarkable use in the theory of extended formulations of polytopes (see for example, \cite{Yannakakis}, \cite{FPTdW}, \cite{Rothvoss}, \cite{GPT2}, \cite{LeeSDP}). 
Their utility in creating a realization space model for polytopes was also observed in \cite{D14}.

\subsection{Our contribution}
By passing to a symbolic version of the slack matrix $S_P$, wherein we replace every positive entry by a distinct variable in the vector of variables $\xx$, one gets a symbolic matrix $S_P(\xx)$. The slack ideal $I_P$ is the ideal obtained by saturating the ideal of $(d+2)$-minors of $S_P(\xx)$ with respect to all variables. The complex variety of $I_P$,
$\mathcal{V}(I_P)$, is the slack variety of $P$. {We prove that modulo a group action, the positive part of $\mathcal{V}(I_P)$  is a realization space for the projective equivalence classes of polytopes that are combinatorially equivalent to $P$. This is the slack realization space of $P$ and it provides a new model for the realizations of a polytope modulo projective transformations. Working with a slightly modified ideal called the {\em affine slack ideal} of $P$, we also obtain a realization space for $P$ that is rationally
equivalent to the classical realization space of $P$. We call this the {\em affine slack realization space} of $P$.} By the positive part of a complex variety we mean the intersection of the variety with the positive real orthant of the ambient space.

The slack realization space has several nice features. The inequalities in its description are simply nonnegativities of variables in place of the determinantal inequalities in the classical model. By forgetting these inequalities one can study the entire slack variety, which is a natural algebraic relaxation of the realization space.
The slack realization space naturally mods out affine equivalence among polytopes and, unlike in the classical construction, does not depend on a choice of affine basis. The construction leads to a natural way to study polytopes up to projective equivalence. Further, it serves as a 
realization space for both the polytope it was constructed from as well as the polar of the polytope.

Additionally, the slack ideal provides a computational engine for establishing several types of results one can ask about the combinatorial class of a polytope. We exhibit three concrete applications
of this machinery to determine non-rationality, non-prescribability of faces, and non-realizability of polytopes.
We expect that further applications and questions on the important and difficult topic of realization spaces will be amenable to 
our algebraic geometry based approach.

\subsection{Organization of the paper}
In Section~\ref{sec:bg}, we summarize the results on slack matrices needed in this paper. We also define the slack ideal and affine slack ideal of a polytope.
In Section~\ref{sec:RealSp}, we construct the slack and affine slack realization spaces of a polytope. We show that the
affine slack realization space is rationally equivalent to the classical realization space of the polytope.
In Section~\ref{sec:Apps} we illustrate how the slack ideal provides a computational framework for many classical questions about polytopes such as convex realizability of combinatorial polytopes, rationality, and prescribability of faces.

\subsection{Acknowledgements} We thank Arnau Padrol and G\"unter Ziegler for helpful pointers to the literature and valuable comments on the first draft of this paper.  The \texttt{SageMath} and \texttt{Macaulay2} software systems were invaluable in the development of the results below. All computations described in this paper were done with one of these two systems \cite{SageMath}, \cite{M2}.


%
%

\section{Background: Slack Matrices and Ideals of Polytopes} \label{sec:bg}

In this section we first present several known results about slack matrices of polytopes needed in this paper.  Many of these results come from \cite{slackmatrixpaper}. We then recall the slack ideal of a polytope from \cite{GPRT17} which will be our main computational engine.  While much of this section is background, we also present new objects and results that play an important role in later sections.

Suppose we are given a polytope $P \subset \RR^d$ with { with vertices labelled $1,\ldots, v$ and facet inequalities labelled $1,\ldots, f$.}
Assume that $P$ is a $d$-polytope, meaning that $\dim(P)=d$.
Recall that $P$ has two usual representations: a $\VV$-representation $P = \conv\{\mb{p}_1,\ldots, \mb{p}_v\}$ as the convex hull of vertices, and an $\mathcal{H}$-representation $P = \{\xx\in\RR^d : W\xx \leq \mb{w}\}$ as the common intersection of the half spaces defined by the facet inequalities $W_j \xx \leq {w}_j$, $j=1,\ldots, f$, where $W_j$ denotes the $j$th row of $W \in \RR^{f \times d}$.
Let $V \in \RR^{v \times d}$ be the matrix with rows ${\mb{p}_1}^\top,\ldots, {\mb{p}_v}^\top$,  and let
$\mathbbm{1}$ denote a vector (of appropriate size) with all entries equal to $1$. Then
the combined data of the two representations yields a {\em slack matrix} of $P$, defined as

\begin{equation} \label{EQ:slackdef} S_P := \left[\begin{array}{cc} \mathbbm{1} & V\\
 \end{array}\right] \left[\begin{array}{c} \mb{w}^\top \\ -W^\top \end{array}\right] \in\RR^{v\times f}. \end{equation}

 The name comes from the fact that the $(i,j)$-entry of $S_P$ is ${w}_j - W_j\mb{p}_i$ which is the {\em slack} of the
 $i$th vertex $\mb{p}_i$ of $P$ with respect to the $j$th facet inequality $W_j \xx \leq w_j$ of $P$. Since $P$ is a $d$-polytope, $\textup{rank}(\left[\begin{array}{cc} \mathbbm{1} & V\\
 \end{array}\right]) =  \rank(\begin{bmatrix} \mb{w} & -W\end{bmatrix}) = d+1$, and hence, $\rank(S_P) = d+1$. Also, $\mathbbm{1}$ is in the column span of $S_P$. While the $\mathcal{V}$-representation of $P$ is unique, the  $\mathcal{H}$-representation is not, as each facet inequality $W_j\xx \leq {w}_j$ is equivalent to the scaled inequality $\lambda W_j\xx \leq \lambda {w}_j \text{ for }\lambda >0$,
and hence $P$ has infinitely many slack matrices obtained by positive scalings of the columns of $S_P$.
 Let $D_t$ denote a diagonal matrix of size $t \times t$ with all positive diagonal entries.
 Then all slack matrices of $P$  are of the form $S_P D_f$ for some $D_f$.

A polytope $Q$ is {\em affinely equivalent} to $P$ if there exists an invertible affine transformation $\psi$
such that $Q = \psi(P)$ {and $\psi$ respects the vertex labelling; that is $\psi(\mb{p}_i) = \mb{q}_i$}. If $Q$ is affinely equivalent to $P$, then $S_P$ is a slack matrix of $Q$ and thus $P$ and $Q$ have the same slack matrices (see Example~\ref{EG:quadslack}).
In fact, a slack matrix of $P$ offers a representation of the affine equivalence class of $P$ by the following result.

\begin{lemma}[{\cite[Theorem~14]{slackmatrixpaper}}] If $S$ is any slack matrix of $P$, then the
polytope $Q = \conv(\textup{rows}(S)),$ is
affinely equivalent to $P$. \label{LEM:rowrealiz} \end{lemma}

By the above discussion, we may translate $P$ so that
$0\in\text{int}(P)$ without changing its slack matrices.
Subsequently, we may scale facet inequalities to set $\mb{w} = \mathbbm{1}$.
Then the affine equivalence class of $P$ can be associated to the slack matrix
\begin{equation} \label{EQ:slackrepdef} S^1_P = [\mathbbm{1}\; V]\left[\begin{array}{c}\mathbbm{1}\\ -W^\top\end{array}\right]
\end{equation}
which has the special feature that the all-ones vector of the appropriate size is present in both its row space and column space.
Again this matrix is not unique as it depends on the position of $0 \in \textup{int}(P)$.

 Recall that the polar of $P$ is $P^\circ  = \{ \yy \in (\RR^d)^\ast \,:\, \langle \xx, \yy \rangle \leq 1 \,\,\,\forall \,\, \xx \in P \}$. Under the assumption that $0\in\text{int}(P)$ and that $\mb{w} = \mathbbm{1}$, $P^\circ$ is again a polytope with $0$ in its interior and representations \cite[Theorem 2.11]{Ziegler}:
 $$P^\circ = \conv\{ W_1^\top, \ldots, W_f^\top\} = \{ \yy \in (\RR^d)^\ast \,:\, V \yy \leq \mathbbm{1} \}.$$
This implies that $(S_P^1)^\top$ is a slack matrix of $P^\circ$ and all slack matrices of $P^\circ$ are of the form $(D_v S_P^1)^\top$.

We now pass from the fixed polytope $P$ to its combinatorial class. Note that the zero-pattern in a slack matrix of $P$, or equivalently, the support of $S_P$, encodes the vertex-facet incidence structure of $P$, and hence the entire combinatorics (face lattice) of $P$. 
A labelled polytope $Q$ is {\em combinatorially equivalent} to $P$ if
$P$ and $Q$ have {the same face lattice 
under the identification of vertex $\mb{p}_i$ in $P$ with vertex $\mb{q}_i$ in $Q$ and the identification of facet inequality $f_j$ in $P$ with facet inequality $g_j$ in $Q$. }
The {\em combinatorial class} of $P$ is the set of all labelled polytopes that are combinatorially equivalent to $P$.
A {\em realization} of $P$ is a polytope $Q$, embedded in some $\RR^k$, that is combinatorially equivalent to $P$.
By our labelling assumptions, all realizations of $P$ have slack matrices with the same support as $S_P$.
Further, since each realization $Q$ of $P$ is again a $d$-polytope, all its slack matrices have rank $d+1$ and contain $\mathbbm{1}$ in their column span. Interestingly, the converse is also true and is a consequence of \cite[Theorem 22]{slackmatrixpaper}.

\begin{theorem}
A nonnegative matrix $S$ is a slack matrix of some realization of the {labelled} $d$-polytope $P$ if and only if all of the following hold:
\begin{enumerate}
\item $\textup{supp}(S) = \textup{supp}(S_P)$ \label{EQ:support}
\item $\rank(S) = \rank(S_P) = d+1$ \label{EQ:rank}
\item $\mathbbm{1}$ lies in the column span of $S$. \label{EQ:colspan}
\end{enumerate}
\label{THM:slackconditions}\end{theorem}

This theorem will play a central role in this paper. It allows us to identify the combinatorial class of $P$
with the set of nonnegative matrices having the three listed properties.

A polytope $Q$ is {\em projectively equivalent} to $P$ if there exists a projective transformation
$\phi$ such that $Q = \phi(P)$ {and $\phi$ respects the vertex labelling; that is, $\phi(\mb{p}_i) = \mb{q}_i$}.
Recall that a projective transformation is a map
$$ \phi: \RR^d \to \RR^d, \,\,\, 	\xx  \mapsto \frac{B\xx+\mathbf{b}}{\mathbf{c}^\top \xx + \gamma}$$
for some $B\in \RR^{d \times d}$, $\mathbf{b,c}\in \RR^d$, $\gamma\in \RR$ such that
\begin{equation} \det\left[\begin{array}{cc} B & \mathbf{b} \\ \mathbf{c}^\top & \gamma \end{array}\right] \neq 0. \end{equation}
The polytopes $P$ and $Q = \phi(P)$ are combinatorially equivalent. Projective equivalence within a combinatorial class can be characterized in terms of slack matrices.

\begin{lemma}[{\cite[Corollary 1.5]{GPRT17}}] Two {labelled} polytopes $P$ and $Q$ are projectively equivalent if and only if $D_vS_PD_f$ is a slack matrix of $Q$ for some positive diagonal matrices $D_v,D_f$. \label{lem:PEscaling} \end{lemma}

Notice that Lemma~\ref{lem:PEscaling} does not say that {\em every} positive scaling of rows and columns of $S_P$ is a slack matrix of a polytope projectively equivalent to $P$, but rather that there is {\em some} scaling of rows and columns of $S_P$ that produces a slack matrix of $Q$. In particular, condition~\eqref{EQ:colspan} of Theorem~\ref{THM:slackconditions} requires $\mathbbm{1}$ to be in the column span of the scaled matrix.  Not all row scalings will preserve $\mathbbm{1}$ in the column span. Regardless, we will be interested in all row and column scalings of slack matrices.

\begin{definition}
A {\em generalized slack matrix} of $P$ is any matrix of the form $D_v S_Q D_f$, where
$Q$ is a polytope that is combinatorially equivalent to $P$ and $D_v, D_f$ are diagonal matrices with positive entries on the diagonal. Let $\mathfrak{S}_P$ denote the set of all generalized slack matrices of $P$.
\end{definition}

\begin{theorem} \label{THM:generalized slack matrices satisfy (1) and (2)}
The set $\mathfrak{S}_P$ of generalized slack matrices of $P$ consists precisely of the nonnegative matrices that satisfy
conditions~\eqref{EQ:support} and~\eqref{EQ:rank} of Theorem~\ref{THM:slackconditions}.
\end{theorem}

\begin{proof}
By construction, every matrix in $\mathfrak{S}_P$ satisfies  conditions~\eqref{EQ:support} and~\eqref{EQ:rank} of Theorem~\ref{THM:slackconditions}.
To see the converse, we need to argue that if $S$ is a nonnegative matrix that
satisfies conditions~\eqref{EQ:support} and~\eqref{EQ:rank} of Theorem~\ref{THM:slackconditions}, then there exists some $D_v,D_f$ such that $S = D_v S_Q D_f$ for some polytope $Q$ that is combinatorially equivalent to $P$, or equivalently, that there is some row scaling of $S$ that turns it into a slack matrix of a polytope combinatorially equivalent to $P$. By Theorem~\ref{THM:slackconditions}, this is equivalent to showing that $\mathbbm{1}$ lies in the column span of $D_v^{-1} S$.
Choose the diagonal matrix $D_v^{-1}$ so that $D_v^{-1} S$ divides each row of $S$ by the
sum of the entries in that row. Note that this operation is well-defined, as a row of all zeros would correspond to a vertex which is part of every facet. Then the sum of the columns of $D_v^{-1} S$ is $\mathbbm{1}$ making
$D_v^{-1} S$ satisfy all three conditions of Theorem~\ref{THM:slackconditions}. Therefore, by the theorem,
$D_v^{-1} S = S_Q$ for some polytope $Q$ in the combinatorial class of $P$.
\end{proof}

We illustrate the above results on a simple example.

\begin{example} \label{EG:quadslack}
Consider two realizations of a quadrilateral in $\RR^2$,
\begin{align*}
P_1 & = \conv\{(0,0),(1,0),(1,1),(0,1)\}, \textup{ and }  \\
P_2 & = \conv\{(1,-2),(1,2),(-1,2),(-1,-2)\},
\end{align*}
where $P_2 = \psi(P_1)$ for the affine transformation $\renewcommand{\arraystretch}{0.7}\psi(\xx) =
\begin{bmatrix} 0 & -2 \\ 4 & 0 \end{bmatrix}
\xx +
\begin{bmatrix}1\\-2\end{bmatrix}$. The most obvious choice of facet representation for $P_1$ yields the slack matrix
\begin{align*}
S_{P_1} =
\begin{bmatrix} 1 & 0 & 0 \\ 1 & 1 & 0 \\ 1 & 1 & 1 \\ 1 & 0 & 1 \end{bmatrix}
\begin{bmatrix} 0 & 1 & 1 & 0 \\ 0 & -1 & 0 & 1 \\ 1 & 0 & -1 & 0 \end{bmatrix}
& =
\begin{bmatrix} 0 & 1 & 1 & 0 \\ 0 & 0 & 1 & 1 \\ 1 & 0 & 0 & 1 \\ 1 & 1 & 0 & 0 \end{bmatrix},
\end{align*}
which, by calculating the effect of $\psi$ on the facets of $P_1$, one finds is the same as the slack matrix for $P_2$,
\begin{align*}
S_{P_2} & =
\begin{bmatrix} 1 & 1 & -2 \\ 1 & 1 & 2 \\ 1 & -1 & 2 \\ 1 & -1 & -2 \end{bmatrix}
\begin{bmatrix} \frac{1}{2} & \frac{1}{2} & \frac{1}{2} & \frac{1}{2} \\ -\frac{1}{2} & 0 & \frac{1}{2} & 0 \\ 0 & -\frac{1}{4} & 0 & \frac{1}{4}
\end{bmatrix}. 
\end{align*}

Since $P_2$ also contains the origin in its interior, we can scale each column of its $\mathcal{H}$-representation from above by 2 to obtain a slack matrix of the form $S_{P_2}^1$.
Finally, consider the following nonnegative matrix
$$S = \begin{bmatrix} 0 & 1 & 1 & 0 \\ 0 & 0 & 1 & 1 \\ 1 & 0 & 0 & 2 \\ 1 & 2 & 0 & 0 \end{bmatrix}.$$
Since $S$ satisfies all three conditions of Theorem~\ref{THM:slackconditions}, it must be the slack matrix of some realization of a quadrilateral. In fact, it is easy to check that $S$ is
the slack matrix of the quadrilateral with vertices $\{(0,0),(1,0),(2,1),(0,1)\}$. Since all quadrilaterals are projectively equivalent, by Lemma~\ref{lem:PEscaling} we must be able to obtain $S_{P_1}$ by scaling the columns and rows of $S$ and, in fact, multiplying its first column by $2$ and its last two rows by $1/2$ we recover $S_{P_1}$.
\end{example}

We now recall the {\em symbolic slack matrix} and {\em slack ideal} of $P$ which were defined in \cite{GPRT17}.
Given a $d$-polytope $P$, its {\em symbolic slack matrix} $S_P(\xx)$ is the sparse generic matrix obtained by replacing each nonzero entry of $S_P$ by a distinct variable. Suppose there are $t$ variables in $S_P(\xx)$. The {\em slack ideal} of $P$ is the saturation of the ideal generated by the $(d+2)$-minors of $S_P(\xx)$, namely
\begin{equation} I_P := \langle (d+2)\text{-minors of }S_P(\xx) \rangle :\left(\prod_{i=1}^t x_i\right)^\infty \subset \CC[\xx] := \CC[x_1, \ldots, x_t]. \label{EQ:slackidealdef} \end{equation}
The {\em slack variety} of $P$ is the complex variety
$\mathcal{V}(I_P) \subset \CC^t$. The saturation of $I_P$ by the product of all variables guarantees that there are no components
in $\mathcal{V}(I_P)$ that live entirely in coordinate hyperplanes.
If $\s \in \CC^t$ is a zero of $I_P$, then we identify it with the matrix
$S_P(\s)$.

\begin{lemma} \label{lem:generalized slack matrices are in the slack variety}
The set $\mathfrak{S}_P$ of generalized slack matrices is contained in the real part of the slack variety
$\mathcal{V}(I_P)$.
\end{lemma}

\begin{proof}
By Theorem~\ref{THM:generalized slack matrices satisfy (1) and (2)}, all matrices in
$\mathfrak{S}_P$ have real entries, support equal to $\supp(S_P)$, and
rank $d+1$. Therefore, $\mathfrak{S}_P$ is contained in the real part of $\mathcal{V}(I_P)$.
\end{proof}

To focus on ``true slack matrices'' of polytopes in the combinatorial class of $P$, meaning matrices that satisfy all conditions of Theorem~\ref{THM:slackconditions},
we define the {\em affine slack ideal}
\begin{equation} \widetilde{I}_P = \langle (d+2)\text{-minors of }[S_P(\xx)\;\mathbbm{1}] \rangle :\left(\prod_{i=1}^t x_i\right)^\infty  \subset \CC[\xx],\label{EQ:trueslackidealdef} \end{equation}
where $[S_P(\xx)\;\mathbbm{1}]$ is the symbolic slack matrix with a column of ones appended.
By construction, $\mathcal{V}(\widetilde{I}_P)$ is a subvariety of $\mathcal{V}(I_P)$.

\begin{definition} Let $\widetilde{\mathfrak{S}}_P$ denote the set of true slack matrices of polytopes in the combinatorial class of $P$, or equivalently, the set of all nonnegative matrices that satisfy the
three conditions of Theorem~\ref{THM:slackconditions}.
\end{definition}

\begin{lemma} \label{lem:true slack matrices are in the true slack variety}
The set $\widetilde{\mathfrak{S}}_P$ of true slack matrices is contained in the real part of $\mathcal{V}(\widetilde{I}_P)$.
\end{lemma}

\begin{proof} By definition, all elements $S\in\widetilde{\mathfrak{S}}_P$ have real entries and $\textup{supp}(S)= \textup{supp}(S_P)$. It remains to show that $\rank([S \; \mathbbm{1}])\leq d+1$. This follows immediately from the fact that $S$ satisfies properties \eqref{EQ:rank} and \eqref{EQ:colspan} of Theorem~\ref{THM:slackconditions}.
\end{proof}

\begin{example} \label{EG:quadideal}
For our quadrilateral $P_1$ from Example~\ref{EG:quadslack} and in fact any quadrilateral $P$
labelled in the same way as $P_1$, we have
$$S_{P}(\xx) = \begin{bmatrix}
0 & x_1 & x_2 & 0 \\
0 & 0 & x_3 & x_4 \\
x_5 & 0 & 0 & x_6  \\
x_7 & x_8 & 0 & 0
\end{bmatrix}.$$
Its slack ideal is
\begin{align*}
{I}_P & = \langle 4\text{-minors of }S_{P}(\xx)\rangle : \left(\prod_{i=1}^8 x_i\right)^\infty
		 = \langle  x_2x_4x_5x_8 - x_1x_3x_6x_7\rangle \subset \CC[x_1,\ldots,x_8].
\end{align*}
The affine slack ideal of $P$ is
\begin{align*}
\widetilde{I}_P  = \langle 4\text{-minors of }[S_{P}(\xx)\;\mathbbm{1}]\rangle : \left(\prod_{i=1}^8 x_i\right)^\infty
	& = \langle  x_1x_3x_6-x_2x_4x_8+x_2x_6x_8-x_3x_6x_8, \\[-7pt]
&\phantom{= \langle}x_2x_4x_5-x_2x_4x_7+x_2x_6x_7-x_3x_6x_7, \\
&\phantom{= \langle}x_1x_4x_5-x_1x_4x_7+x_1x_6x_7-x_4x_5x_8, \\
&\phantom{= \langle}x_1x_3x_5-x_1x_3x_7+x_2x_5x_8-x_3x_5x_8\rangle.
\end{align*}

Notice, for example, that the generalized slack matrix which corresponds to $\s = (2,2,2,1,8,2,2,1)$ is a zero of $I_P$ but not of $\widetilde{I}_P$ and indeed $\mathbbm{1}$ is not in the column span of $S_P(\s)$. \qed
\end{example}


\begin{example} In the above example, the slack ideal $I_P$ is the same as the ideal of minors before saturation. To see that saturation can have an effect on the ideal, we consider the example of the following symbolic slack matrix, \[
S_P(\xx)=\begin{bmatrix}
    0 &    x_1 &      0 &      0 &      0 &    x_2 &      0 \\
  x_3 &      0 &      0 &      0 &      0 &    x_4 &      0 \\
  x_5 &      0 &    x_6 &      0 &      0 &      0 &    x_7 \\
    0 &    x_8 &    x_9 &      0 &      0 &      0 & x_{10} \\
    0 &      0 &      0 &      0 & x_{11} &      0 & x_{12} \\
    0 &      0 &      0 & x_{13} & x_{14} & x_{15} &      0 \\
    0 &      0 & x_{16} & x_{17} &      0 &      0 &      0
\end{bmatrix}
\]
which belongs to the four-dimensional polytope with $f$-vector $(7,17,17,7)$ shown in \cite[line 3, Table 1]{GPRT17}.
The ideal of $6$-minors of the above symbolic slack matrix has 49 generators. 
All generators are all binomial except the following 4:
$$
\begin{array}{c}x_{2}x_{3}x_{7}x_{8}x_{13}x_{16}-x_{1}x_{4}x_{5}x_{10}x_{13}x_{16}-x_{1}x_{3}x_{7}x_{9}x_{15}x_{17}+x_{1}x_{3}x_{6}x_{10}x_{15}x_{17} \\
x_{2}x_{3}x_{6}x_{8}x_{12}x_{14}-x_{1}x_{4}x_{5}x_{9}x_{12}x_{14}+x_{1}x_{3}x_{7}x_{9}x_{11}x_{15}-x_{1}x_{3}x_{6}x_{10}x_{11}x_{15}\\
x_{4}x_{5}x_{10}x_{11}x_{13}x_{16}-x_{4}x_{5}x_{9}x_{12}x_{14}x_{17}+x_{3}x_{7}x_{9}x_{11}x_{15}x_{17}-x_{3}x_{6}x_{10}x_{11}x_{15}x_{17}\\
-x_{2}x_{7}x_{8}x_{11}x_{13}x_{16}+x_{2}x_{6}x_{8}x_{12}x_{14}x_{17}+x_{1}x_{7}x_{9}x_{11}x_{15}x_{17}-x_{1}x_{6}x_{10}x_{11}x_{15}x_{17}\end{array}.$$
After saturation, the slack ideal has only 9 generators, all of which are binomial, as listed below:
$$\begin{array}{lll} 
x_7x_9-x_6x_{10}, & x_4x_5x_{12}x_{14}-x_3x_7x_{11}x_{15},  & x_7x_{11}x_{13}x_{16}-x_6x_{12}x_{14}x_{17}, \\
x_2x_3x_6x_8-x_1x_4x_5x_9, & x_4x_5x_{13}x_{16}-x_3x_6x_{15}x_{17}, &x_2x_8x_{12}x_{14}-x_1x_{10}x_{11}x_{15}, \\ x_2x_3x_7x_8-x_1x_4x_5x_{10}, &x_2x_8x_{13}x_{16}-x_1x_9x_{15}x_{17}, & x_{10}x_{11}x_{13}x_{16}-x_9x_{12}x_{14}x_{17}\end{array}$$
We postpone further discussion of why such a simplification of the slack ideal will be beneficial until the end of the next section. 
\qed
\end{example}


%
%

\section{Realization spaces from Slack Varieties} \label{sec:RealSp}

Recall that a realization of a $d$-polytope $P \subset \RR^d$ is a polytope $Q$ that is combinatorially equivalent to $P$. A {\em realization space} of $P$ is, essentially, the set of all polytopes $Q$ which are realizations of $P$, or equivalently, the set of all ``geometrically distinct'' polytopes which are combinatorially equivalent to $P$.
We say ``essentially'' since it is typical to mod out by affine equivalence within the combinatorial class.

The standard construction of a realization space of $P = \conv\{\mb{p}_1,\ldots, \mb{p}_v\}$ is as follows (see \cite{RG96}). Fix an {\em affine basis} of $P$, that is, $d+1$ vertex labels $B = \{b_0,\ldots, b_d\}$ such that the vertices $\{\mb{p}_{b}\}_{b\in B}$ are necessarily affinely independent in every realization of $P$. Then the {\em realization space of $P$ with respect to $B$} is
$$\mathcal{R}(P,B) = \{\text{realizations $Q = \conv\{\mb{q}_1,\ldots,\mb{q}_v\}$ of $P$ with $\mb{q}_i=\mb{p}_i$ for all $i\in B$}\}.$$
Fixing an affine basis ensures that just one $Q$ from each affine equivalence class in the combinatorial class of $P$ occurs in $\mathcal{R}(P,B)$.

Realization spaces of polytopes are {\em primary basic semialgebraic sets}, that is, they are defined by finitely many polynomial equations and strict inequalities. Recording each realization $Q$ by its vertices, we can think of $\mathcal{R}(P,B)$ as lying in $\RR^{d \cdot v}$.
Two primary basic semialgebraic sets $X\subseteq\RR^m$ and $Y\subseteq\RR^{m+n}$ 
are {\em rationally equivalent} if there exists a homeomorphism $f:X\to Y$ such that both $f$ and $f^{-1}$ are rational functions.
The important result for us is that if $B_1,B_2$ are two affine bases of a polytope $P$, then $\mathcal{R}(P,B_1)$ and $\mathcal{R}(P,B_2)$ are rationally equivalent \cite[Lemma 2.5.4]{RG96}. Thus one can call $\mathcal{R}(P,B) \subset \RR^{d \cdot v}$, {\em the} realization space of $P$.

The main goal of this section is to construct models of realization spaces for $P$ from the slack variety $\VV(I_P) \subset \CC^t$ and affine slack variety $\VV(\widetilde{I}_P)$ defined in Section~\ref{sec:bg}.
Recall that we identify an element $\s$ in either variety with the
matrix $S_P(\s)$.
Then by Lemma~\ref{lem:generalized slack matrices are in the slack variety},
$\mathfrak{S}_P$,  the set of all generalized slack matrices of all polytopes in the combinatorial class of $P$, is contained in $\VV(I_P)$. Similarly, by Lemma~\ref{lem:true slack matrices are in the true slack variety}, $\widetilde{\mathfrak{S}}_P$, the set of all true slack matrices of polytopes in the combinatorial class of $P$, is contained in $\VV(\widetilde{I}_P)$.
In fact, $\mathfrak{S}_P$  is contained in the positive part of $\VV(I_P)$, defined as 
\begin{equation} \label{EQ:slackmodel} \VV_+(I_P) := \VV(I_P)\cap \RR^t_{>0}\end{equation}
and $\widetilde{\mathfrak{S}}_P$
is contained in the positive part of $\VV(\widetilde{I}_P)$ defined as
\begin{equation} \label{EQ:slackmodel} \VV_+(\widetilde{I}_P) := \VV(\widetilde{I}_P)\cap \RR^t_{>0}.\end{equation}

These positive spaces are going to lead to realization spaces of $P$. In order to get there, we first describe these sets more explicitly. We start with a well-known lemma, whose proof we include for later reference.

\begin{lemma} \label{lem:matrices have correct rank}
Let $S$ be a matrix with the same support as $S_P$. Then $\rank(S) \geq d+1$.
\end{lemma}

\begin{proof}
Consider a flag of $P$, i.e., a maximal chain of faces in the face lattice of $P$.
Choose a sequence of facets $F_0,F_1, \ldots, F_d$ so that the flag is
$$\emptyset \;= \; F_0\cap\cdots\cap F_d  \;\subset\; F_1\cap\cdots\cap F_d \;\subset\; \cdots \;\subset\; F_{d-1}\cap F_d \;\subset\; F_d \;\subset\; P.$$
Next choose a sequence of vertices so that $v_0 = F_1\cap\cdots\cap F_d$ is the $0$-face in the flag, making $v_0 \not \in F_0$. Then choose $v_1 \in F_2 \cap \dots \cap F_d$ but $v_1 \not \in F_1$, $v_2 \in F_3 \cap \dots \cap F_d$ but $v_2 \not \in  F_2$ and so on, until $v_{d-1} \in F_d$ but not in $F_{d-1}$. Finally, choose $v_d$ so that $v_d \not \in F_d$.
Then the $(d+1)\times(d+1)$ submatrix of $S_P$ indexed by the chosen vertices and facets is lower triangular with a nonzero diagonal, hence has rank $d+1$.

Now if $S$ is a matrix with $\supp(S) = \supp(S_P)$, $S$ will also have this lower triangular submatrix in it, thus $\rank(S) \geq d+1$.
\end{proof}

We remark that the vertices chosen from the flag in the above proof form a suitable affine basis to fix in the
construction of $\mathcal{R}(P,B)$.

\begin{theorem} \label{THM:description of positive slack variety}
The positive part of the slack variety, $\VV_+(I_P)$, coincides with $\mathfrak{S}_P$, the set of generalized slack matrices of $P$. Similarly, $\VV_+(\widetilde{I}_P)$ coincides with $\widetilde{\mathfrak{S}}_P$, the set of true slack matrices of $P$.
\end{theorem}

\begin{proof}
We saw that $\mathfrak{S}_P \subseteq \VV_+(I_P)$ and by Theorem~\ref{THM:generalized slack matrices satisfy (1) and (2)}, $\mathfrak{S}_P$ is precisely the set of nonnegative matrices with the same support as $S_P$ and rank $d+1$.
On the other hand, if $\s \in \VV_+(I_P)$, then $S_P(\s)$ is nonnegative and $\supp(S_P(\s)) = \supp(S_P)$. Therefore, by Lemma~\ref{lem:matrices have correct rank}, $\rank(S_P(\s)) = d+1$. Thus,
$\VV_+(I_P) = \mathfrak{S}_P$.

We saw that  $\widetilde{\mathfrak{S}}_P \subseteq \VV_+(\widetilde{I}_P)$.
Also recall that $\VV_+(\widetilde{I}_P)$ is contained in $\VV_+(I_P)$. Therefore, by the first statement of the theorem, if $\s \in \VV_+(\widetilde{I}_P)$, then $S_P(\s)$ is nonnegative,
$\supp(S_P(\s)) = \supp(S_P)$ and $\rank(S_P(\s)) = d+1$. From the definition of $\widetilde{I}_P$, we have $\rank([S_P(\s)\ \mathbbm{1}]) \leq d+1$, so it follows that $\rank([S_P(\s)\ \mathbbm{1}]) = d+1$, or equivalently, $\mathbbm{1}$ lies in the column span of $S_P(\s)$. Therefore, the matrices in $\VV_+(\widetilde{I}_P)$ satisfy all three conditions of Theorem~\ref{THM:slackconditions}, hence $\VV_+(\widetilde{I}_P) =  \widetilde{\mathfrak{S}}_P$.
\end{proof}

Since positive row and column scalings of a generalized slack matrix of $P$ give another generalized slack matrix of $P$, we immediately get that $\VV_+(I_P)$ is closed under row and column scalings.
Similarly, $\VV_+(\widetilde{I}_P)$ is closed under column scalings.

\begin{corollary} \label{cor:scale}\
\begin{enumerate}
\item If $\s\in\VV_+(I_P)$, then $D_v\s D_f\in \VV_+(I_P)$, for all positive diagonal matrices $D_v,D_f$.
\item Similarly, if  $\s\in\VV_+(\widetilde{I}_P)$, then $\s D_f\in \VV_+(\widetilde{I}_P)$, for all positive diagonal matrices $D_f$.
\end{enumerate}
\end{corollary}

Corollary~\ref{cor:scale} tells us that the groups $\RR_{>0}^v\times\RR_{>0}^f$ and $\RR^f_{>0}$ act on $\VV_+(I_P)$ and $\VV_+(\widetilde{I}_P)$, respectively, via multiplication by positive diagonal matrices. Modding out these actions is the same as setting some choice of variables in the symbolic slack matrix to $1$, which means that we may choose a representative of each equivalence class (affine or projective) with ones in some prescribed positions.

\begin{corollary} \label{THM:realizespmodel}\
\begin{enumerate}
\item Given a polytope $P$, there is a bijection
between the elements of\break $\VV_+(I_P)/(\RR^v_{>0}\times\RR^f_{>0})$ and the classes of projectively equivalent polytopes of the same combinatorial type as $P$.
In particular, each class contains a true slack matrix.

\item Given a polytope $P$, there is a bijection
between the elements of $\VV_+(\widetilde{I}_P)/\RR^f_{>0}$ and the classes of affinely equivalent polytopes of the same combinatorial type as $P$.
\end{enumerate}
\end{corollary}

The last statement in Corollary~\ref{THM:realizespmodel} (1) follows from the fact that every generalized slack matrix admits a row scaling that makes it satisfy all three conditions of Theorem~\ref{THM:slackconditions}, thereby making it a true slack matrix. An explicit example of such a scaling can be seen in the proof of Theorem~\ref{THM:generalized slack matrices satisfy (1) and (2)}. { Recall that slack matrices are only defined up to positive column scalings. Statement (2) above then follows from the fact that affinely equivalent polytopes have the same slack matrices.}

By the above results we have that $\VV_+(I_P)/(\RR^v_{>0}\times\RR^f_{>0})$ and $\VV_+(\widetilde{I}_P)/\RR^f_{>0}$ are parameter spaces for the projective (respectively, affine) equivalence classes of polytopes in the combinatorial class of $P$. Thus they can be thought of as realization spaces of $P$.

\begin{definition} \label{def:realization spaces}
Call $\VV_+(I_P)/(\RR^v_{>0}\times\RR^f_{>0})$ the {\em slack realization space} of the polytope $P$, and
$\VV_+(\widetilde{I}_P)/\RR^f_{>0}$ the {\em affine slack realization space} of the polytope $P$.
\end{definition}

We will see below that the affine slack realization space $\VV_+(\widetilde{I}_P)/\RR^f_{>0}$ is rationally equivalent to the
classical model of realization space $\mathcal{R}(P,B)$ of the polytope $P$.  On the other hand, our main object, the slack realization space $\VV_+(I_P)/(\RR^v_{>0}\times\RR^f_{>0})$, does not have an analog in the polytope literature.
This is partly because in every realization of $P$, fixing a projective basis does not guarantee that the remaining vertices in the realization are not at infinity. The slack realization space is a natural model for the realization space of projective equivalence classes of polytopes. {We note that in \cite{GPS17} the authors investigate the projective realization space of combinatorial hypersimplices and find an upper bound for its dimension. However they do not present an explicit model for it.}

\begin{theorem} \label{THM:truerealizequiv}
The affine slack realization space $\VV_+(\widetilde{I}_P)/\RR^f_{>0}$ is rationally equivalent to the
classical realization space $\mathcal{R}(P,B)$ of the polytope $P$.
\end{theorem}

\begin{proof} We will show that $\VV_+(\widetilde{I}_P)/\RR^f_{>0}$ is rationally equivalent to
$\mathcal{R}(P,B)$ for a particular choice of $B$. By \cite[Lemma 2.5.4]{RG96}, this is sufficient to show rational equivalence for any choice of basis.

We have already shown that realizations of $P$ modulo affine transformations are in bijective correspondence with the elements of both
$\VV_+(\widetilde{I}_P)/\RR^f_{>0}$ and $\mathcal{R}(P,B)$. So we just have to prove that this bijection induces a rational equivalence between these spaces, i.e., both the map and its inverse are rational.

We will start by showing the map sending a polytope in $\mathcal{R}(P,B)$ to its slack matrix is rational.
Fix a flag in $P$, as in the proof of Lemma~\ref{lem:matrices have correct rank}.
Suppose the sequence of vertices and facets chosen from the flag in the proof are indexed by the sets $I$ and $J$ respectively. The vertices $\{\mb{p}_i\}_{i\in I}$ are affinely independent, so that $B = I$ is an affine basis of $P$. Moreover, by applying an affine transformation to $P$, we may assume that $0$ is in the convex hull of $\{\mb{p}_i\}_{i\in I}$, hence is in the interior of every element of $\mathcal{R}(P,B)$.   Consider the map
\[
g:\mathcal{R}(P,B) \to \VV_+(\widetilde{I}_P), \,\,\,\,\,  Q \mapsto S_Q^1.
\]
The polytope $Q$ is recorded in $\mathcal{R}(P,B)$ by its list of vertices, which in turn are the rows of the matrix $V$. Also, recall that $S^1_Q = [\mathbbm{1}\; V]\begin{bmatrix} \mathbbm{1} \\ -W^\top \end{bmatrix}$. To prove that $g$ is a rational map, we need to show that the matrix of facet normals $W$ is a rational function of $V$. Since we know the combinatorial type of $P$, we know the set of vertices that lie on each facet. For facet $j$, let $V(j)$ be the submatrix of $V$ whose rows are the vertices on this facet.
Then the normal of facet $j$, or equivalently $W_j$, is obtained by solving the linear system
$V(j) \cdot \xx  = \mathbbm{1}$
which proves that $W_j$ is a rational function of $V$.
Then $\widetilde{g} = \pi\circ g$ is the desired rational map from $\mathcal{R}(P,B)$ to $\VV_+(\widetilde{I}_P)/\RR^f_{>0}$, where $\pi$ is the standard quotient map $\pi: \VV_+(\widetilde{I}_P)\to\VV_+(\widetilde{I}_P)/\RR^f_{>0}$. It sends the representative in $\mathcal{R}(P,B)$ of an affine equivalence class of polytopes in the combinatorial class of $P$ to the representative of that class in $\VV_+(\widetilde{I}_P)/\RR^f_{>0}$.

For the reverse map, we have to send a slack matrix $S_Q$ of a realization $Q$ of $P$ to the representative of its affine equivalence class in $\mathcal{R}(P,B)$.
We saw in Lemma~\ref{LEM:rowrealiz} that the rows of $S_Q$ are the vertices of a realization $Q'$ of $P$ that is affinely equivalent to $Q$. So
we just have to show that
$Q'$ can be rationally mapped to the representative of $Q$ in $\mathcal{R}(P,B)$. To do that, denote by $\widehat{S}_Q$ the $(d+1) \times (d+1)$ lower triangular submatrix of $S_Q$ from our flag, with rows indexed by $I$ and columns indexed by $J$. Then $(\widehat{S}_Q)^{-1}$ consists of rational functions in the entries of $\widehat{S}_Q$.
Let $\mathcal{B}$ be the $(d+1)\times d$ matrix whose rows are the vertices of $P$ indexed by $B$. Recall that these vertices are common to all elements of $\mathcal{R}(P,B)$, and in particular, they form an affine basis for the representative of $Q$ in $\mathcal{R}(P,B)$.
Then the linear map
$$\psi_{S_Q}:\RR^{f}\to \RR^d, \,\,\,\, \xx \mapsto  \xx_J^\top\widehat{S}_Q^{-1}\mathcal{B},$$
 where $\xx_J$ is the restriction of $\xx \in \RR^f$ to the coordinates indexed by $J$,
 is defined rationally in terms of the entries of $S_Q$, and maps row $i$ of $S_Q$ to the affine basis vertex $\mb{p}_i$, for all $i \in I$.
Now since $\psi_{S_Q}$ is a linear map, $\psi_{S_Q}(Q')$ is affinely equivalent to $Q'$ which is itself affinely equivalent to $Q$. Furthermore, $\psi_{S_Q}$ sends an affine basis of $Q'$ to the corresponding affine basis in $Q$, so in fact it must be a bijection between the two polytopes.
 Hence, $\psi_{S_Q}(\textup{rows of }S_Q)$ equals the representative of $Q$ in $\mathcal{R}(P,B)$, completing our proof.
\end{proof}

The slack realization space is especially elegant in the context of polarity. 
Let $P^\circ$ be the polar polytope of $P$. It is not immediately obvious from the standard model of a realization space, how $\mathcal{R}(P,B_1)$ and $\mathcal{R}(P^\circ,B_2)$ are related. In \cite{RG96}, it is shown that the realization spaces of $P$ and $P^\circ$ are {\em stably equivalent}, a coarser notion of equivalence than rational equivalence (see \cite[Definition 2.5.1]{RG96}); however, the proof of this fact in Theorem 2.6.3 is non-trivial.
Now consider the slack model. Recall we know that one slack matrix of $P^\circ$ is $(S_P^1)^\top$, so that $S_{P^\circ}(\xx) = S_P(\xx)^\top$. In particular, this means that $I_{P^\circ} = I_P$, so that the slack varieties and realization spaces of $P$ and $P^\circ$ are actually the same when considered as subsets of $\RR^t$. We simply need to interpret $\s\in\VV_+(I_P) = \VV_+(I_{P^\circ})$ as a realization of $P$ or $P^\circ$ by assigning its coordinates to $S_P(\xx)$ along rows or columns.

\begin{example} \label{EG:quadrealize} Let us return to the realization space of the unit square
$P_1$ from Example~\ref{EG:quadslack}. Suppose we fix the affine basis $B = \{1,2,4\}$, where we had $\mb{p}_1 = (0,0), \mb{p}_2 = (1,0)$ and $\mb{p}_4=(0,1)$. 
Then the classical realization space $\mathcal{R}(P_1,B)$ consists of all quadrilaterals $Q = \conv\{\mb{p}_1,\mb{p}_2,(a,b), \mb{p}_4\}$, where $a,b\in\RR$ must satisfy $a,b>0$ and $a+b>1$ in order for $Q$ to be convex.

In the slack realization spaces, modding out by row and column scalings is equivalent to fixing some
variables in $S_P(\xx)$ to $1$. So for example, we could start with the following scaled symbolic slack and affine slack matrices
$$S_P(\xx) = \begin{bmatrix}
0 & 1 & 1 & 0 \\
0 & 0 & 1 & 1 \\
1 & 0 & 0 & 1  \\
1 & x_8 & 0 & 0
\end{bmatrix}, \hspace{10pt}
[S_P(\xx)\;\mathbbm{1}] = \begin{bmatrix}
0 & 1 & 1 & 0 & 1 \\
0 & 0 & x_3 & x_4 & 1\\
1 & 0 & 0 & 1 & 1 \\
x_7 & x_8 & 0 & 0 & 1
\end{bmatrix}.$$
Computing the $4$-minors of these scaled symbolic slack matrices and saturating with all
variables produces the scaled slack ideals
\begin{align*}
I_P^{\textup{scaled}}& = \langle  x_8 - 1\rangle, \textup{ and }  \\
\widetilde{I}_P^{\textup{scaled}}  & = \langle  x_3x_8 + x_4x_8 - x_3 - x_8, \,x_4x_7 + x_4x_8 - x_4 - x_7, \,x_3x_7 - x_4x_8\rangle.
\end{align*}

Therefore the slack realization space, $\VV_+(I_P)/ ({\RR^4_{>0} \times \RR^4_{>0}})$, has the unique element $(1,1,1,1,1,1,1,1)$, and
indeed, all convex quadrilaterals are projectively equivalent to $P_1$. 

From the generators of  $\widetilde{I}_P^{\textup{scaled}}$ one sees that the affine slack realization space,
$\VV(\widetilde{I}_P)/ \RR^4_{>0}$, is two-dimensional and parametrized by $x_3,x_4$ with
$$ x_7 =  \frac{x_4}{x_3 + x_4 -1} \,\,\textup{ and } \,\,x_8 = \frac{x_3}{x_3 + x_4 -1}.$$
Since all the four variables have to take on positive values in a (scaled) slack matrix of a quadrilateral, we get that
the realization space
$\VV_+(\widetilde{I}_P)/\RR^4_{>0}$ is cut out by the inequalities $ x_3 >  0, \,\, x_4 > 0, \,\, x_3 + x_4 > 1$.
This description coincides exactly with that of $\mathcal{R}(P_1,B)$ that we saw earlier.
\label{EX:square}
 \end{example}

 \begin{example} \label{ex:vertex split of vertex sum}
Consider the $5$-polytope $P$ with vertices $\mb{p}_1,\ldots, \mb{p}_8$ given by
\begin{align*}
e_1,e_2,e_3,e_4,-e_1-2e_2-e_3,-2e_1-e_2-e_4,-2e_1-2e_2+e_5,-2e_1-2e_2-e_5
\end{align*}
where $e_1,\ldots, e_5$ are the standard basis vectors in $\RR^5$. It can be obtained by {splitting} the distinguished vertex $v$ of the vertex sum of two squares, $(\Box,v)\oplus(\Box,v)$ in the notation of \cite{McMull}. This polytope has 8 vertices and 12 facets and its symbolic slack matrix has the zero-pattern below
$$\begin{bmatrix}
0&*&0&0&0&0&*&0&0&0&0&0 \\
0&0&0&*&*&0&0&0&0&0&0&0 \\
0&0&0&0&0&*&*&*&0&0&*&* \\
*&0&0&*&0&*&0&0&*&0&*&0    \\
*&*&*&0&0&0&0&0&*&*&0&0    \\
0&0&*&0&*&0&0&*&0&*&0&*  \\
*&0&*&0&0&*&0&*&0&0&0&0 \\
0&0&0&0&0&0&0&0&*&*&*&*
\end{bmatrix}.$$
By \cite[Theorem 5.3]{McMull} $P$ is not projectively unique, meaning its slack realization space $\VV_+(I_P)/(\RR_{>0}^v\times\RR^f_{>0})$ will consist of more than a single point. Indeed, by fixing ones in the maximum number of positions, {marked in \red{bold} face below}, we find that $\VV_+(I_P)/(\RR_{>0}^v\times\RR^f_{>0})$ is a one-dimensional space of projectively inequivalent realizations parametrized by slack matrices of the following form
$$S_P(a) = \begin{bmatrix}
0&\bf{\red{1}}&0&0&0&0&\bf{\red{1}}&0&0&0&0&0 \\
0&0&0&\bf{\red{1}}&\bf{\red{1}}&0&0&0&0&0&0&0 \\
0&0&0&0&0&\bf{\red{1}}&\bf{\red{1}}&1&0&0&1&\bf{\red{1}} \\
\bf{\red{1}}&0&0&\bf{\red{1}}&0&\bf{\red{1}}&0&0&a&0&a&0    \\
\bf{\red{1}}&1&\bf{\red{1}}&0&0&0&0&0&{1}&1&0&0    \\
0&0&1&0&\bf{\red{1}}&0&0&\bf{\red{1}}&0&a&0&a  \\
1&0&1&0&0&1&0&\bf{\red{1}}&0&0&0&0 \\
0&0&0&0&0&0&0&0&\bf{\red{1}}&\bf{\red{1}}&\bf{\red{1}}&\bf{\red{1}}
\end{bmatrix}.$$
If we wish to look at a representative of each equivalence class which is a true slack matrix, then we can scale the above to guarantee that $\mathbbm{1}$ is in the column space.

\label{EQ:Amys5d}
\end{example}


In the next section we will use $I_P$ to obtain geometric results about the realization space $\VV_+(I_P )$. 
The positive part of the real variety of the ideal of $(d+2) \times (d+2)$ minors of $S_P({\bf x})$ coincides with $\VV_+(I_P)$. 
However, the variety of this determinantal ideal typically has 
a number of extraneous components that do not correspond to true polytopes. In particular, it may contain 
degenerate components contained entirely in coordinate hyperplanes, corresponding to matrices with too many 
zeroes to be slack matrices. The presence of such components results in an unnecessarily complicated algebraic description 
of the realization space. Saturating the determinantal ideal removes these degenerate parts, and the resulting simpler ideal 
$I_P$ more faithfully captures the geometry of the realization space. However, $I_P$ is still not necessarily the vanishing ideal 
of $\VV_+(I_P )$. Indeed, one can see in \cite[Section~5]{GMTWsecondpaper} that $\VV_+(I_P )$ need not be 
Zariski dense in $\VV(I_P )$. However, enlarging $I_P$ further would require much heavier and impractical computations, 
and it is unclear if it would have much impact on the actual algebraic description.

\section{Applications}
\label{sec:Apps}

In this section we illustrate the computational power of the slack ideal in answering three types
of questions that one can ask about realizations of polytopes. We anticipate further applications.

\subsection{Abstract polytope with no realizations}
Checking if an abstract polytopal complex is the boundary of an actual polytope is the classical
{\em Steinitz problem}, and an important ingredient in cataloging polytopes with few vertices. In \cite{alts85}, Altshuler and Steinberg enumerated all $4$-polytopes and $3$-spheres with $8$~vertices.
The first non-polytopal $3$-sphere in  \cite[Table 2]{alts85} has simplices and square pyramids as facets, and these facets have the following vertex sets
$$12345, 12346, 12578, 12678, 14568, 34578, 2357, 2367, 3467, 4678.$$

If there was a polytope $P$ with these facets, its symbolic slack matrix would be
$$S_P(\xx)=\begin{bmatrix}
    0  &        0 &       0 &       0 &       0 &   x_1   &    x_2  &    x_3   & x_4    & x_5      \\
    0  &        0 &       0 &       0 &     x_6 &   x_7   &      0  &      0   & x_8    & x_9      \\
    0  &        0 &  x_{10} &  x_{11} &  x_{12} & 0       &      0  &      0   &   0    & x_{13}   \\
    0  &        0 &  x_{14} &  x_{15} &  0      & 0       & x_{16}  & x_{17}   &   0    & 0        \\
    0  &   x_{18} &       0 &  x_{19} &  0      & 0       &      0  & x_{20}   & x_{21} & x_{22}   \\
x_{23} &        0 &  x_{24} &       0 &  0      & x_{25}  & x_{26}  &      0   &   0    & 0        \\
x_{27} &   x_{28} &       0 &  0      &  x_{29}      & 0       &      0  &      0   &   0    & 0        \\
x_{30} &   x_{31} &       0 &       0 &  0      & 0       & x_{32}  & x_{33}   & x_{34} & 0        \\
\end{bmatrix}.$$
One can compute that the would-be slack ideal $I_P$ in this case is trivial, meaning that there is no rank five matrix with the support of $S_P(\xx)$. In particular, there is no polytope with the given facial structure. In fact,
there is not even a hyperplane-point arrangement in $\RR^4$ or $\CC^4$ with the given incidence structure.

In some other cases, one can obtain non-empty slack varieties that have no positive part. A simple example of that behaviour can be seen in the {\em tetrahemihexahedron}, a polyhedralization of the real projective plane with $6$ vertices, and facets with vertex sets $235$, $346$, $145$, $126$, $2456$, $1356$, $1234$. Its slack matrix is therefore
$$S_P(\xx)=\begin{bmatrix}
   x_{1}& x_{2}&   0&   0&   x_{3}&    0&   0\\
   0&  x_{4}&  x_{5}&   0&    0&   x_{6}&   0\\
   0&   0&  x_{7}&  x_{8}&   x_{9}&    0&   0\\
 x_{10}&   0&   0& x_{11}&    0&  x_{12}&   0\\
   0& x_{13}&   0& x_{14}&    0&    0& x_{15}\\
 x_{16}&   0& x_{17}&   0&    0&    0& x_{18}
\end{bmatrix}.$$
Computing the slack ideal from the 5-minors of $S_P(\xx)$ we find that $I_P$ is generated by the binomials

\begin{small}
\begin{center}
\begin{tabular}{ccc}
$x_{8}x_{15}x_{17} + x_{7}x_{14}x_{18}$ &
$x_{4}x_{15}x_{17} + x_{5}x_{13}x_{18}$&
$x_{11}x_{15}x_{16} + x_{10}x_{14}x_{18}$\\
$x_{2}x_{15}x_{16} + x_{1}x_{13}x_{18}$&
$x_{5}x_{12}x_{16} + x_{6}x_{10}x_{17}$&
$x_{7}x_{11}x_{16} - x_{8}x_{10}x_{17}$\\
$x_{3}x_{7}x_{16} + x_{1}x_{9}x_{17}$&
$x_{2}x_{5}x_{16} - x_{1}x_{4}x_{17}$&
$x_{6}x_{11}x_{13} + x_{4}x_{12}x_{14}$\\
$x_{1}x_{11}x_{13} - x_{2}x_{10}x_{14}$&
$x_{5}x_{8}x_{13} - x_{4}x_{7}x_{14}$&
$x_{3}x_{8}x_{13} + x_{2}x_{9}x_{14}$\\
$x_{6}x_{7}x_{11} + x_{5}x_{8}x_{12}$&
$x_{3}x_{8}x_{10} + x_{1}x_{9}x_{11}$&
$x_{2}x_{6}x_{10} + x_{1}x_{4}x_{12}$
 \end{tabular}

 \begin{tabular}{cc}
$x_{6}x_{11}x_{15}x_{17} - x_{5}x_{12}x_{14}x_{18}$& $x_{2}x_{9}x_{15}x_{17} - x_{3}x_{7}x_{13}x_{18}$\\ 
$x_{4}x_{12}x_{15}x_{16} - x_{6}x_{10}x_{13}x_{18}$& $x_{3}x_{8}x_{15}x_{16} - x_{1}x_{9}x_{14}x_{18}$\\ 
$x_{2}x_{7}x_{14}x_{16} - x_{1}x_{8}x_{13}x_{17}$& $x_{5}x_{11}x_{13}x_{16} - x_{4}x_{10}x_{14}x_{17}$\\ 
$x_{2}x_{6}x_{9}x_{11} - x_{3}x_{4}x_{8}x_{12}$& $x_{2}x_{5}x_{8}x_{10} - x_{1}x_{4}x_{7}x_{11}$\\ 
$x_{3}x_{6}x_{7}x_{10} - x_{1}x_{5}x_{9}x_{12}$& 
$x_{3}x_{4}x_{7} + x_{2}x_{5}x_{9}$\\
 \end{tabular}
\end{center}
\end{small}

 Since the slack ideal contains binomials whose coefficients are both positive, it has no positive zeros. In fact, by fixing some coordinates to one, it has a unique zero up to row and column scalings, where all entries are either $1$ or $-1$.

\subsection{Non-prescribable faces of polytopes}

Another classical question about polytopes is whether a face can be freely prescribed in a realization of a polytope with given combinatorics.

We begin by observing that there is a natural relationship between the slack matrix/ideal of a polytope and those of each of its faces. For instance, if $F$ is a facet of a $d$-polytope $P$, a symbolic slack matrix $S_F(\xx)$ of $F$ is the submatrix of $S_P(\xx)$ indexed by the vertices of $F$ and the facets of $P$ that intersect $F$ in its $(d-2)$-dimensional faces. Let $\xx_F$ denote the vector of variables in that submatrix.
All $(d+1)$-minors of $S_F(\xx)$ belong to the slack ideal $I_P$. To see this,
consider a $(d+2)$-submatrix of $S_P(\xx)$ obtained by enlarging the given $(d+1)$-submatrix of $S_F(\xx)$ by a row indexed by a vertex $\mb{p} \not \in F$ and the column indexed by $F$. The column of $F$ in this bigger submatrix has all zero entries except in position $(\mb{p},F)$. The minor of this $(d+2)$-submatrix in $S_P(\xx)$ after saturating out the variable in position $(\mb{p},F)$, is the $(d+1)$-minor of $S_F(\xx)$ that we started with.
Therefore, $$I_F \subseteq I_P \cap \CC[\xx_F].$$
By induction on the dimension, this containment is true for all faces $F$ of $P$.

A face $F$ of a polytope $P$ is \textit{prescribable} if, given any realization of $F$, we can complete it to a realization of $P$. In our language, a face $F$ is prescribable in $P$ if and only if
$$\VV_+(I_F)=\VV_+(I_P \cap \CC[\xx_F]).$$

Consider the four-dimensional prism over a square pyramid, {for which it was shown in \cite{Barn87}
that its only cube facet $F$ is non-prescribable}. This polytope $P$ has $10$ vertices and $7$ facets and its symbolic slack matrix is
$$S_P(\xx)=\begin{bmatrix}
  \bf{\red{x_1}}& \bf{\red{0 }}    &   0   & \bf{\red{0 }}    & \bf{\red{x_{2}}} & \bf{\red{x_{3}}} & \bf{\red{ 0}}\\
  \bf{\red{x_{4}}}& \bf{\red{0 }}     &   0   & \bf{\red{0 }}    & \bf{\red{  0}}    & \bf{\red{x_{5}}} &\bf{\red{x_{6}}}\\
  \bf{\red{x_{7}}}& \bf{\red{0  }}    &   0   & \bf{\red{x_{8}}}&  \bf{\red{ 0}}    & \bf{\red{  0}}    &\bf{\red{x_{9}}}\\
  \bf{\red{x_{10}}}& \bf{\red{0 }}     &   0   & \bf{\red{x_{11}}}& \bf{\red{x_{12}}} & \bf{\red{  0}}    & \bf{\red{ 0}}\\
  x_{13}     & 0           & x_{14}&   0        &   0         &   0         &  0\\
  \bf{\red{0}}    & \bf{\red{x_{15}}}  &   0   & \bf{\red{0}}     & \bf{\red{x_{16}}} & \bf{\red{x_{17}}} & \bf{\red{ 0}}\\
  \bf{\red{0}}    & \bf{\red{x_{18}}}  &   0   & \bf{\red{0}}     & \bf{\red{  0}}    & \bf{\red{x_{19}}} &\bf{\red{x_{20}}}\\
  \bf{\red{0}}    & \bf{\red{x_{21}}}  &   0   & \bf{\red{x_{22}}}& \bf{\red{  0}}    &  \bf{\red{ 0}}    &\bf{\red{x_{23}}}\\
  \bf{\red{0}}    & \bf{\red{x_{24}} } &   0   & \bf{\red{x_{25}}}&\bf{\red{ x_{26}}} & \bf{\red{  0}}    & \bf{\red{ 0}}\\
  0         & x_{27}      &x_{28} &   0        &   0         &   0         &  0
\end{bmatrix}.$$
In \red{bold} we mark $S_F(\xx)$ sitting inside $S_P(\xx)$. Computing $I_P$ and intersecting with $\CC[\xx_F]$, we obtain an ideal
of dimension $15$. On the other hand, the slack ideal of a cube has dimension $16$, suggesting an extra degree of freedom for the realizations of a cube, and the possibility that the cubical facet $F$ cannot be arbitrarily prescribed in a realization of $P$. However, we need more of an argument to conclude this, since
$I_F \neq I_P \cap \CC[\xx_F]$ does not immediately mean that $\VV_+(I_F) \neq \VV_+(I_P \cap \CC[\xx_F])$.
We need to compute further to get Barnette's result.

We first note that one can scale the rows and columns of $S_F(\xx)$ to set $13$ of its $24$ variables to one, say
$x_1,x_2,x_3,x_4,x_6,x_7,x_8,x_{10},x_{15},x_{16}, x_{18},x_{21}, x_{24}.$ Guided by the resulting slack ideal we further set
$x_{20}=1, x_{11}=\frac{1}{2}, x_{17} = 2$ and $x_{25}=1$. Now solving for the remaining variables from the equations of the slack ideal, we get the following true slack matrix of a cube:
$$\begin{bmatrix}
1&0&0&1&1&0\\
1&0&0&0&2&1\\
1&0&1&0&0&1\\
1&0&1/2&1&0&0\\
0&1&0&1&2&0\\
0&1&0&0&3&1\\
0&1&3/2&0&0&1\\
0&1&1&1&0&0
\end{bmatrix}.$$
However, making the above-mentioned substitutions for
$$x_1,x_2,x_3,x_4,x_6,x_7,x_8,x_{10},x_{11},x_{15},x_{16}, x_{17},x_{18}, x_{20},x_{21}, x_{24},x_{25}$$
in $S_P(\xx)$ and eliminating $x_{13},x_{14},x_{27}$ and $x_{28}$ from the slack ideal results in the trivial ideal showing that the cube on its own admits further realizations than are possible as a face of $P$.

\subsection{Non-rational polytopes}
A combinatorial polytope is said to be \textit{rational} if it has a realization in which all vertices have rational entries. This has a very simple interpretation in terms of slack varieties.

\begin{lemma} \label{lem:rational}
A polytope $P$ is rational if and only if $\VV_+(I_P)$ has a rational point.
\end{lemma}

The proof is trivial since any rational realization gives rise to a rational slack matrix and any rational slack matrix is itself a rational realization of the polytope $P$. Recall that any point in $\VV_+(I_P)$ can be row scaled to be a true slack matrix by dividing each row by the sum of its entries, so a rational point in $\VV_+(I_P)$ will provide a true rational slack matrix of $P$.

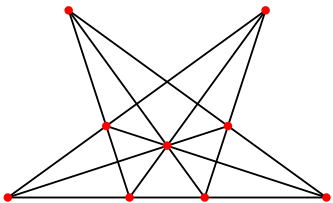
\begin{figure}
\begin{tikzpicture}[scale=0.5,line cap=round,line join=round,line width=.7pt]
\clip(-4.1,-0.4) rectangle (6.2,5.5);
{\tikzstyle{every node}=[circle,draw=red,fill=red,inner sep=0pt,minimum width=2.5pt]
\draw (0.,0.)-- (2.,0.);
\draw (2.,0.)-- (2.618033988749895,1.9021130325903064);
\draw (2.618033988749895,1.9021130325903064)-- (1.,3.077683537175253);
\draw (1.,3.077683537175253)-- (-0.6180339887498947,1.9021130325903073);
\draw (-0.6180339887498947,1.9021130325903073)-- (0.,0.);
\draw (-1.6180339887498945,4.979796569765561)-- (2.,0.);
\draw (3.6180339887498953,4.9797965697655595)-- (0.,0.);
\draw (5.236067977499787,0.)-- (-0.6180339887498947,1.9021130325903073);
\draw (-3.2360679774997907,0.)-- (2.618033988749895,1.9021130325903064);
\draw (-1.6180339887498945,4.979796569765561)-- (-0.6180339887498947,1.9021130325903073);
\draw (-1.6180339887498945,4.979796569765561)-- (1.,3.077683537175253);
\draw (1.,3.077683537175253)-- (3.6180339887498953,4.9797965697655595);
\draw (3.6180339887498953,4.9797965697655595)-- (2.618033988749895,1.9021130325903064);
\draw (2.,0.)-- (5.236067977499787,0.);
\draw (2.618033988749895,1.9021130325903064)-- (5.236067977499787,0.);
\draw (-3.2360679774997907,0.)-- (0.,0.);
\draw (-3.2360679774997907,0.)-- (-0.6180339887498947,1.9021130325903073);
\node at (0.,0.) {};
\node at (2.,0.) {};
\node at (2.618033988749895,1.9021130325903064) {};
\node at (-0.6180339887498947,1.9021130325903073) {};
\node at (-1.6180339887498945,4.979796569765561) {};
\node at (3.6180339887498953,4.9797965697655595) {};
\node at (5.236067977499787,0.) {};
\node at (-3.2360679774997907,0.) {};
\node at (1,1.38) {};}
\end{tikzpicture}
\caption{Non-rational line-point configuration}
\label{fig:pointconfig}
\end{figure}

Unfortunately, the usual examples of non-rational polytopes tend to be too large for direct computations, so we illustrate our point on the
non-rational point-line arrangement in the plane shown in Figure~\ref{fig:pointconfig} from 
\cite[Figure~5.5.1]{Grunbaum}.
A true non-rational polytope can be obtained from this point-line arrangement by Lawrence lifting. We will show the non-rationality of this configuration by computing its slack ideal as if it were a $2$-polytope. Its symbolic slack matrix is the $9\times 9$ matrix
$$S(\xx)=\begin{bmatrix}
 x_{1}&   {0}& x_{2}&   {0}& x_{3}& x_{4}& x_{5}& x_{6}& {0}  \\
 x_{7}& x_{8}& x_{9}&   {0}&x_{10}&   {0}&   {0}&x_{11}&x_{12}\\
x_{13}&x_{14}&   {0}&x_{15}&x_{16}&x_{17}&x_{18}&   {0}& {0}  \\
x_{19}&x_{20}&   {0}&x_{21}&   {0}&   {0}&x_{22}&x_{23}&x_{24}\\
x_{25}&   {0}&x_{26}&x_{27}&   {0}&x_{28}&   {0}&   {0}&x_{29}\\
   {0}&   {0}&x_{30}&x_{31}&x_{32}&   {0}&x_{33}&x_{34}&x_{35}\\
   {0}&x_{36}&   {0}&x_{37}&x_{38}&x_{39}&   {0}&x_{40}&x_{41}\\
   {0}&x_{42}&x_{43}&   {0}&x_{44}&x_{45}&x_{46}&   {0}&x_{47}\\
   {0}&x_{48}&x_{49}&x_{50}&   {0}&x_{51}&x_{52}&x_{53}& {0}
\end{bmatrix}.$$
One can simplify the computations by scaling rows and columns to fix $x_i=1$ for $i=1,2,8,14,20,26,30,36,42,44,47,48,50,51,52,53$, as this does not affect rationality. Then
one sees that the polynomial $x_{46}^2 + x_{46} - 1$ is in the slack ideal, so $x_{46}=\frac{-1\pm\sqrt{5}}{2}$, and there are no rational realizations of this configuration.

\begin{remark}
We note that as illustrated by the above example, the slack matrix and slack ideal constructions are not limited to the setting of polytopes, but in fact, are applicable to the more general setting of any point/hyperplane configuration. In particular, the slack realization space model can be extended to the setting of matroids as in \cite{BW18}.
\end{remark}

\bibliographystyle{alpha}
\bibliography{all}
\end{document}